\documentclass[12pt]{article}
\usepackage{amssymb,amsthm,amsmath,amsfonts,latexsym,tikz,hyperref,enumitem}
\usepackage[hmargin=1in,vmargin=1in]{geometry}

\newtheorem{thm}{Theorem}[section]
\newtheorem{prop}[thm]{Proposition}
\newtheorem{cor}[thm]{Corollary}
\newtheorem{lem}[thm]{Lemma}
\newtheorem{conj}[thm]{Conjecture}
\newtheorem{exa}[thm]{Example}
\newtheorem{question}[thm]{Question}

\newcommand{\cU}{\mathcal{U}}

\newcommand{\phib}{\overline{\phi}}

\newcommand{\tA}{\tilde{A}}
\newcommand\pair[2]{\genfrac{}{}{0pt}{1}{#1}{#2}}
\newcommand\dn[1]{\lfloor#1\rfloor}  
\newcommand\up[1]{\lceil#1\rceil}  
\newcommand\Dn[1]{\left\lfloor#1\right\rfloor}
\newcommand\Up[1]{\left\lceil#1\right\rceil}
\newcommand{\cIb}{\ol{\cI}}
\newcommand{\cUb}{\ol{\cU}}
\newcommand{\cIr}{\cI^r}
\newcommand{\cUr}{\cU^r}
\newcommand{\cIGr}{\cI^r}
\newcommand{\cUGr}{\cU^r}

\DeclareMathOperator{\cen}{cen}

\newcommand{\ben}{\begin{enumerate}}
\newcommand{\een}{\end{enumerate}}
\newcommand{\ble}{\begin{lem}}
\newcommand{\ele}{\end{lem}}
\newcommand{\bth}{\begin{thm}}
\renewcommand{\eth}{\end{thm}}
\newcommand{\bpr}{\begin{prop}}
\newcommand{\epr}{\end{prop}}
\newcommand{\bco}{\begin{cor}}
\newcommand{\eco}{\end{cor}}
\newcommand{\bcon}{\begin{conj}}
\newcommand{\econ}{\end{conj}}
\newcommand{\bde}{\begin{defn}}
\newcommand{\ede}{\end{defn}}
\newcommand{\bex}{\begin{exa}}
\newcommand{\eex}{\end{exa}}
\newcommand{\barr}{\begin{array}}
\newcommand{\earr}{\end{array}}
\newcommand{\btab}{\begin{tabular}}
\newcommand{\etab}{\end{tabular}}
\newcommand{\beq}{\begin{equation}}
\newcommand{\eeq}{\end{equation}}
\newcommand{\bea}{\begin{eqnarray*}}
\newcommand{\eea}{\end{eqnarray*}}
\newcommand{\bal}{\begin{align*}}
\newcommand{\bce}{\begin{center}}
\newcommand{\ece}{\end{center}}
\newcommand{\bpi}{\begin{picture}}
\newcommand{\epi}{\end{picture}}
\newcommand{\bpp}{\begin{picture}}
\newcommand{\epp}{\end{picture}}
\newcommand{\bfi}{\begin{figure} \begin{center}}
\newcommand{\efi}{\end{center} \end{figure}}
\newcommand{\bprf}{\begin{proof}}
\newcommand{\eprf}{\end{proof}\medskip}
\newcommand{\capt}{\caption}
\newcommand{\bsl}{\begin{slide}{}}
\newcommand{\esl}{\end{slide}}
\newcommand{\bfr}{\begin{frame}}
\newcommand{\efr}{\end{frame}}

\newcommand{\comp}{\models}

\newcommand{\hqed}{\hfill \qed}

\newcommand{\ol}{\overline}

\newcommand{\hso}[1]{\hspace{-1pt}}

\newcommand{\qmq}[1]{\quad\mbox{#1}\quad}

\newcommand{\lt}{\lhd}
\newcommand{\gt}{\rhd}
\newcommand{\lte}{\unlhd}

\newcommand{\fl}[1]{\lfloor #1 \rfloor}
\newcommand{\ce}[1]{\lceil #1 \rceil}

\def\<{\langle}
\def\>{\rangle}

\newcommand{\spn}[1]{\langle{#1}\rangle}

\newcommand{\ra}{\rightarrow}

\newcommand{\al}{\alpha}
\newcommand{\be}{\beta}

\newcommand{\de}{\delta}

\newcommand{\1}{{\bf 1}}

\newcommand{\cI}{{\cal I}}

\newcommand{\cO}{{\cal O}}

\newcommand{\Fb}{\ol{F}}

\newcommand{\Phib}{\ol{\Phi}}

\newcommand{\rb}{\ol{r}}

\def\multiset#1#2{\ensuremath{\left(\kern-.3em\left(\genfrac{}{}{0pt}{}{#1}{#2}\right)\kern-.3em\right)}}

\DeclareMathOperator{\rk}{rk}

\DeclareMathOperator{\st}{st}

\begin{document}
\pagestyle{plain}

\title{Partial rank symmetry of distributive lattices for fences
}
\author{
Sergi Elizalde\\[-5pt]
\small Department of Mathematics,
Dartmouth College,\\[-5pt]
\small Hanover, NH 03755-3551, USA,
\texttt{sergi.elizalde@dartmouth.edu}\\
Bruce E. Sagan\\[-5pt]
\small Department of Mathematics, Michigan State University,\\[-5pt]
\small East Lansing, MI 48824-1027, USA, {\tt sagan@math.msu.edu}
}

\maketitle

\begin{abstract}
Associated with any composition $\be=(a,b,\ldots)$ is a corresponding fence poset $F(\be)$ whose covering relations are
$$
x_1\lt x_2 \lt \ldots\lt x_{a+1}\gt x_{a+2}\gt \ldots\gt x_{a+b+1}\lt x_{a+b+2}\lt \ldots\ .
$$
The distributive lattice $L(\be)$ of all lower order ideals of $F(\be)$ is important in the theory of cluster algebras. In addition, its rank generating function $r(q;\be)$ is used to define $q$-analogues of rational numbers.  O\u{g}uz and Ravichandran recently showed that its coefficients satisfy an interlacing condition, proving a conjecture of McConville, Smyth and Sagan, which in turn implies a previous conjecture of Morier-Genoud and Ovsienko that $r(q;\be)$ is unimodal.  We show that, when $\be$ has an odd number of parts, then the polynomial is also partially symmetric: the number of ideals of $F(\be)$ of size $k$ equals the number of filters of size $k$, when $k$ is below a certain value.  Our proof is completely bijective.  O\u{g}uz and Ravichandran also introduced a circular version of fences and proved, using algebraic techniques, that the distributive lattice for such a poset is rank symmetric.  We give a bijective proof of this result as well.
We end with some questions and conjectures raised by this work.
\end{abstract}

\begin{flushleft}
	\small {\bf Keywords:} bottom heavy, bottom interlacing, distributive lattice, fence, gate, order ideal, poset, rank, symmetric, unimodal, top heavy, top interlacing           \\[5pt]
	\small {\bf 2020 AMS subject classification:}  06A07  (Primary) 05A15, 05A19, 05A20, 06D05  (Secondary)
\end{flushleft}


\section{Introduction}

We will be studying the rank sequences for distributive lattices of certain partially ordered sets (posets) called fences, defined as follows.
Any terms or notation from the theory of posets which are not defined here can be found in the texts of Sagan~\cite{sag:aoc} or Stanley~\cite{sta:ec1}.
A {\em chain} of length $l$ is a totally ordered set with $l+1$ elements.
A {\em composition of $m$} is a sequence $\be=(\be_1,\be_2,\ldots,\be_s)$ of positive integers, called {\em parts}, with $\sum_i\be_i = m$.  We write $\be\comp m$.  The corresponding fence $F(\be)$ is obtained by taking chains $S_i$ of length $\be_i$ for $1\le i\le s$ and identifying the maximal (respectively, minimal) elements of $S_i$ and $S_{i+1}$ for $i$ odd (respectively, even).  As an example, the fence $F(2,4,1)$  is displayed in Figure~\ref{F(2,4,1)}.  Placing the chains $S_1,S_2,\dots,S_s$ from left to right as in the figure, we label the elements of $F(\be)$ as $x_1,x_2,\ldots,x_n$ from left to right. We say that $S_i$ is an {\em ascending} or {\em descending} segment of $F(\be)$ depending on whether $i$ is odd or even, respectively.
Note that if $\#F(\be)=n$, where the hash tag denotes cardinality, then $\be\comp n-1$.

\bfi
\begin{tikzpicture}[scale=.8]
\fill(0,1) circle(.1);
\fill(1,2) circle(.1);
\fill(2,3) circle(.1);
\fill(3,2) circle(.1);
\fill(4,1) circle(.1);
\fill(5,0) circle(.1);
\fill(6,-1) circle(.1);
\fill(7,0) circle(.1);
\draw (0,1)--(2,3)--(6,-1)--(7,0);
\draw (0,.5) node{$x_1$};
\draw (1,1.5) node{$x_2$};
\draw (2,2.5) node{$x_3$};
\draw (3,1.5) node{$x_4$};
\draw (4,.5) node{$x_5$};
\draw (5,-.5) node{$x_6$};
\draw (6,-1.5) node{$x_7$};
\draw (7,-.5) node{$x_8$};
\end{tikzpicture}
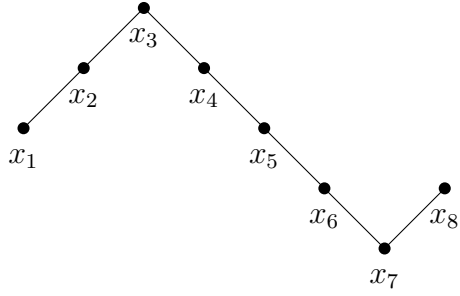
\capt{The fence $F(2,4,1)$ \label{F(2,4,1)}}
\efi

Let $L(\be)$ be the distributive lattice of lower order ideals of $F(\be)$.  These lattices can be used to compute mutations in an associated cluster algebra on a surface with marked points.  In fact there are (at least) six methods for doing so, 
see~\cite{cla:epp,pro:cfp,sch:cef,sch:caa,ST:caa,yur:cef,yur:cce}.
Since $L(\be)$ is ranked, it has an associated {\em rank sequence} $r(\be):r_0,r_1,\ldots,r_n$ where
$$
r_k = \text{number of elements at rank $k$ in $L(\be)$}.
$$
for $0\le k\le n$.  The corresponding rank generating functions
$$
r(q;\be) = \sum_{k=0}^n r_k q^k.
$$
were used by Morier-Genoud and Ovsienko to define $q$-analogues of rational, and even real, numbers~\cite{MGO:qdr}.
For example, for the fence $F(\beta)$ with $\beta=(6,2,1,2,3,1,6)$ (see Figure~\ref{Phi-example}), the rank generating function is
\begin{align*}
r(q;\be) & = 
1+4q+11{q}^{2}+23{q}^{3}+41{q}^{4}+65{q}^{5}+93{q}^{6}+121{q}^{7}+146{q}^{8}+163{q}^{9}+170{q}^{10}\\
&\quad +165{q}^{11}+147{q}^{12}+122{q}^{13}+93{q}^{14}+65{q}^{15}+41{q}^{16}+23{q}^{17}+11{q}^{18}+4{q}^{19}+{q}^{20}.
\end{align*}

Two well-studied properties of sequences $b:b_0,b_1,\ldots,b_n$ are as follows.  Call the sequence {\em symmetric} if
$$
b_k = b_{n-k}
$$
for $0\le k\le n$.   The sequence is said to be {\em unimodal}  if there is an index $m$ such that
$$
b_0\le b_1\le \ldots \le b_m \ge b_{m+1} \ge \ldots \ge b_n.
$$
Sequences satisfying these properties abound in combinatorics, algebra, and geometry.  See the survey articles of Stanley~\cite{sta:lus}, Brenti~\cite{bre:lus}, or Br\"and\'en~\cite{bra:ulr} for examples.
  In their previously cited paper, Morier-Genoud and Ovsienko made the following conjecture which has now been proved, as we will discuss shortly.
\begin{conj}[\cite{MGO:qdr}]
\label{MGO}
For all $\be$, the sequence $r(\be)$ is unimodal.
\end{conj}

It is not true that $r(\be)$ is always symmetric.  For example, when $\be=(1,1)$ we have $r(\be):1,2,1,1$.  However, there are other recently studied properties of sequences~\cite{ath:esl,ath:gcg,BJM:hpz,BS:sdr,SV:uqi,sol:sns}
which are satisfied by $r(\be)$.
Call a sequence $b:b_0,b_1,\ldots,b_n$ {\em top heavy} if
$$
b_k\le b_{n-k}
$$
for $0\le k < \fl{n/2}$, where $\fl{\cdot}$ is the floor (round-down) function.  Dually, the sequence is {\em bottom heavy} if
$$
b_k\ge b_{n-k}
$$
for $0\le k < \fl{n/2}$.  Call the sequence {\em top interlacing} if
$$
b_0\le b_n \le b_1 \le b_{n-1} \le \ldots \le b_{\ce{n/2}},
$$
where $\ce{\cdot}$ is the ceiling (round-up) function.  Top interlacing clearly implies top heavy, and it also gives unimodality, since the inequalities imply that the sequence is increasing up to $b_{\ce{n/2}}$ and decreasing from $b_{\ce{n/2}}$ onward.  Some papers use the term ``alternately increasing,"  but we prefer ``top interlacing'' because it emphasizes how the first and second halves of the sequence interlace.
Similarly, define a sequence to be 
{\em bottom interlacing} if
$$
b_n\le b_0 \le b_{n-1} \le b_1 \le \ldots \le b_{\fl{n/2}}.
$$
As before, bottom interlacing implies both bottom heavy and unimodal.
McConville, Sagan, and Smyth~\cite{MSS:ruc} conjectured the following strengthening of Conjecture~\ref{MGO}, which has has recently been proved by  
O\u{g}uz and Ravichandran~\cite{OR:rpf} using induction and algebraic techniques.

\begin{thm}[\cite{OR:rpf}]
\label{heavy}
Let $\be=(\be_1,\ldots,\be_s)$.
\ben
\item[(a)]  If $s=1$ then $r(\be) = (1,1,\ldots,1)$.
\item[(b)]  If $s$ is even, then $r(\be)$ is bottom interlacing.
\item[(c)]  Suppose $s\ge3$ is odd and let $\be'=(\be_2,\ldots,\be_{s-1})$.
	\ben
	\item[(i)] If $\be_1>\be_s$ then $r(\be)$ is bottom interlacing.
	\item[(ii)] If $\be_1<\be_s$ then $r(\be)$ is top interlacing.
	\item[(iii)] If $\be_1=\be_s$ then $r(\be)$ is symmetric, bottom interlacing, or top interlacing depending on whether 
	$r(\be')$ is symmetric, top interlacing, or bottom interlacing, respectively.
	\een
\een
\end{thm}

The purpose of the present work is to show that, even though $r(\be)$ is not always symmetric, it exhibits at least partial symmetry if there is an odd number of segments.  In particular, our main result is as follows.
\begin{thm}
\label{main}
Let $\be=(\be_1,\be_2,\ldots,\be_s)$ where $s$ is odd and 
$r(\be):r_0,r_1,\ldots,r_n$.  For all $k\le\min\{\be_1,\be_s\}$ we have
$$
r_k = r_{n-k}.
$$
\end{thm}

\bfi
\begin{tikzpicture}[scale=.8]
\fill(0,2) circle(.1);
\fill(1,3) circle(.1);
\fill(2,0) circle(.1);
\fill(2,2) circle(.1);
\fill(3,3) circle(.1);
\fill(4,2) circle(.1);
\draw (0,2)--(1,3)--(2,2)--(3,3)--(4,2)--(2,0)--(0,2);
\draw (0,1.5) node{$x_2$};
\draw (1,2.5) node{$x_3$};
\draw (2,1.5) node{$x_4$};
\draw (3,2.5) node{$x_5$};
\draw (4,1.5) node{$x_6$};
\draw (2,-.5) node{$x_1= x_7$};
\end{tikzpicture}
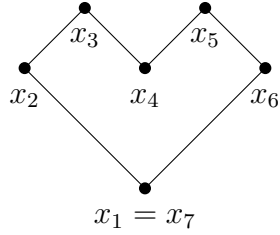
\capt{The circular fence $\Fb(2,1,1,2)$ \label{FbFig}}
\efi

O\u{g}uz and Ravichandran's proof of Theorem~\ref{heavy} relied on certain posets obtained by making the Hasse diagram of a fence into a cycle.
Let $\be=(\be_1,\be_2,\ldots,\be_{2\ell})\comp n$ be a composition with an even number of parts, so that the fence $F(\be)$ has $n+1$ elements $x_1,\ldots,x_{n+1}$, begins with an ascending segment, and ends with a descending segment.
Define the corresponding {\em circular fence} to be the poset $\Fb(\be)$ with $n$ elements obtained by identifying $x_1$ and $x_{n+1}$.  For example, $\Fb(2,1,1,2)$ is displayed in Figure~\ref{FbFig}.  Denote the rank sequence of $\Fb(\be)$ by
$\rb(\be)$.  Using algebraic manipulation of recurrence relations, O\u{g}uz and Ravichandran proved the following result, and left finding a bijective proof as an open problem.
\begin{thm}[\cite{OR:rpf}]
\label{FbSym}
Let $\be=(\be_1,\be_2,\ldots,\be_{s})$ where $s$ is even. Then $\rb(\be)$ is symmetric.
\end{thm}

The rest of this paper will be structured as follows.
In the next section, we will present a totally bijective proof of Theorem~\ref{main}. 
Section~\ref{psc} will be devoted to showing that our bijection can be used, with minor modifications, to prove Theorem~\ref{FbSym} as well.  
We will end with a section of comments and open questions.


\section{Proof of partial symmetry for fences}
\label{pps}

In order to give our bijective proof of Theorem~\ref{main}, we will need some definitions and notation.  In a poset, an {\em ideal} will always be a lower order ideal.  We will also use the terms {\em upper order ideal} and {\em filter}
interchangeably. 
Consider a composition $\be=(\be_1,\be_2,\ldots,\be_{2\ell+1})$ with an odd number of parts.
For a fence $F(\be)$ and $k\ge0$, we let
$$
\cI_k(\be) =\{I \mid \text{$I$ is a lower order ideal of $F(\be)$ with $\#I=k$}\}
$$
and
$$
\cU_k(\be) =\{U \mid \text{$U$ is an upper order ideal of $F(\be)$ with $\#U=k$}\}.
$$
To prove Theorem~\ref{main}, it suffices to construct a bijection $\Phi:\cI_k(\be)\ra\cU_k(\be)$ for all $k\le \min\{\be_1,\be_{2\ell+1}\}$.  This is because, with the notation of the theorem, we have $\#\cI_k(\be)=r_k$ and $\#\cU_k(\be)=r_{n-k}$.

\bfi
\begin{tikzpicture}[scale=.8]
\foreach \position in {(0,4), (1,3), (2,2), (3,3), (4,2), (5,1), (6,0), (7,1), (8,0)}
   \fill\position circle(.1);
\draw (0,4)--(2,2)--(3,3)--(6,0)--(7,1)--(8,0);
\end{tikzpicture}
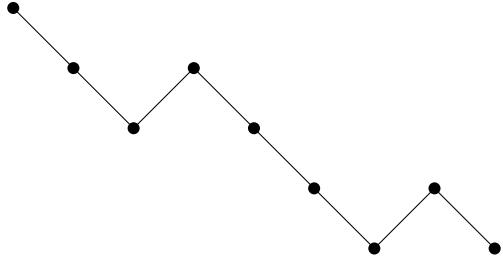
\capt{The gate $G(2,3,1)$. \label{G(2,3,1)}}
\efi

\subsection{Bijection $\phi$ for gates}

To define $\Phi$, we will first construct a bijective map $\phi$ on certain ideals of a particular subposet of a fence obtained by removing the first and last segments, and requiring ascending segments to have length one.
For an arbitrary composition $\de=(\de_1,\de_2,\ldots,\de_{\ell})$, let the corresponding {\em gate} be
$$
G(\de) = F(\de_1,1,\de_2,1,\ldots,\de_{\ell-1},1,\de_\ell)^*,
$$
where the star indicates poset dual. 
The gate $G(2,3,1)$ is shown in Figure~\ref{G(2,3,1)}.
We will use the same terminology for gates as we do for fences.
Note that $G(\de)$ begins and ends with a descending segment. 
Let $D_i$ denote the $i$th descending segment from the left, which has length $\de_i$.
The ideals of a gate which correspond to those of bounded size in the corresponding fence are as follows.  If $G(\de)$ has $\ell$ descending segments, then call an ideal $I$ of this gate {\em restricted} if $\#(I\cap D_1)\le \de_1$ and $\#(I\cap D_\ell)\neq 1$.  In other words, $I$ is restricted if it does not contain the maximal element on $D_1$, and  if it contains the minimal element on $D_\ell$ then it also contains the element above it.
Let
$$
\cIGr(\de)=\{I \mid \text{$I$ is a restricted  ideal of the gate $G(\de)$}\}
$$
Call a filter $U$ of $G(\de)$ {\em restricted} if $\#(U\cap D_1)\neq 1$ 
and $\#(U\cap D_\ell)\leq \de_\ell$.  Equivalently, $U^*$ is a restricted ideal of $G(\de)^*$, which is isomorphic to $G(\de^R)$ where 
$$
\de^R=(\de_\ell,\de_{\ell-1},\ldots,\de_2,\de_1)
$$
is the reversal of $\de$.  In general, the reversal of any sequence $b$ will be denoted by $b^R$.  Note the difference between our use of $r$ for restricted and $R$ for reversal.
The notation for restricted filters is, as expected,
$$
\cUGr(\de)=\{U \mid \text{$U$ is a restricted  filter of the gate $G(\de)$}\}
$$

We will describe a cardinality-preserving bijection 
$$
    \phi:\cIGr(\de)\ra \cUGr(\de).
$$
We will need some more notation and terminology.  Given a sequence 
$d:d_1,d_2,\ldots,d_\ell$, we use the floor symbol
$$
\dn{d}=\dn{d_1,d_2,\ldots,d_\ell}
$$
to denote the subset of $G(\de)$ (if it exists) consisting of the smallest $d_i$ elements on segment $D_i$ for $1\le i\le \ell$.  It is easy to see that $\dn{d}$ exists and is a restricted lower order ideal if and only if the following conditions hold. We use the notation $[m,n]$ for the interval of integers between $m$ and $n$ inclusive, which is shortened to $[n]$ if $m=1$.  The restricted ideal conditions are:
\begin{enumerate}[label={I\arabic*}]
    \item \label{I1} (existence) for $i\in[\ell]$ we have $0\le d_i\le\de_i+1$,
    \item \label{I2} (lower order ideal) for $i\in[2,\ell]$: if $d_i=\de_i+1$ then $d_{i-1}>0$,
    \item \label{I3} (restricted) $d_1\le \de_1$ and $d_\ell\neq1$.
\end{enumerate}
Similarly, we use ceiling notation
$$
\up{e}=\ce{e_1,e_2,\ldots,e_\ell}
$$
to denote the subset of $G(\de)$ containing the largest $e_i$ elements on segment $i$ for $1\le i\le \ell$.  Here are the conditions for $\up{d}$ to exist and be a restricted filter:
\begin{enumerate}[label={U\arabic*}]
    \item \label{U1} (existence) for $i\in[\ell]$ we have $0\le e_i\le\de_i+1$,
    \item \label{U2} (upper order ideal) for $i\in[\ell-1]$:
    if $e_i=\de_i+1$ then $e_{i+1}>0$ for,
    \item \label{U3} (restricted) $e_1\neq1$ and $e_\ell\le \de_\ell$.
\end{enumerate}

A {\em factor} of the sequence $d:d_1,d_2,\ldots,d_\ell$ is a subsequence $d_i,d_{i+1},\ldots,d_j$ of consecutive elements.  If the $d_i$ are nonnegative integers then a {\em block} is a maximal factor of positive integers.  For example, the sequence  
$$
d: 6,1,1,1,0,4,5,1,1,0,0,3,1,2
$$
has three blocks, namely $6,1,1,1$; $4,5,1,1$; and $3,1,2$.  The {\em factor of trailing ones} of a block $B$ is the (possibly empty) maximal factor $T$ of $B$ consisting only of ones such that there is no element of $B$ larger than one to its right.  In our example, the blocks have three, two, and no trailing ones, respectively.  Note that if $\dn{d_1,d_2,\ldots,d_\ell}$ is a restricted ideal, then any nonempty factor $T$ of trailing ones must be followed by a $0$.  This follows directly from the definition of $T$ unless its block contains the last element $d_\ell$.  And in that case, since the ideal is restricted, we must have $d_\ell\ge2$ so that no trailing ones are possible.

One can now construct $\phi(\dn{d_1,d_2,\dots,d_\ell})$ as follows. Consider each block $B$ of the sequence $d_1,d_2,\dots,d_\ell$, and factor it as the concatenation $B=B'T$ where $T$ is $B$'s factor of trailing ones and $B'$ is the rest of $B$.  The map $\phi$  performs the following two steps.
\begin{enumerate}[label={P\arabic*}]
    \item \label{P1} For each nonempty factor $T$ of trailing ones, exchange $T$ with the $0$ to its right.
    \item \label{P2} For each $B'$  with $\#B'\ge2$, decrease the rightmost such entry by $1$ and increase the leftmost one by $1$.
\end{enumerate}

Continuing our example, the three blocks have $3$, $2$ and $0$ trailing ones and $B'$ equal to $6$; $4,5$; and $3,1,2$ from left to right.  So after \ref{P1} we have the sequence
$$
6,0,1,1,1,4,5,0,1,1,0,3,1,2.
$$
Now applying \ref{P2} gives
$$
\phi(\dn{6,1,1,1,0,4,5,1,1,0,0,3,1,2})=\up{6,0,1,1,1,5,4,0,1,1,0,4,1,1}.
$$
Note that the construction of $\phi(\dn{d})$ does not depend on the lengths $\de_i$.

For the following proof, it will be convenient to extend the reversal operator as follows.
If $\dn{d}$ is an ideal of $G(\de)$, then let
$$
\dn{d}^R =\up{d^R},
$$
where $\up{d^R}$ is being considered as a filter of $G(\de^R)$.  Similarly let
$$
\up{e}^R = \dn{e^R}.
$$

\begin{thm} \label{thm:phi}
The map $\phi:\cIGr(\de)\ra \cUGr(\de)$  defined by~\ref{P1} and~\ref{P2} is a cardinality-preserving bijection.
\end{thm}
\begin{proof}
Let $\de=(\de_1,\de_2,\dots,\de_\ell)$, and suppose we are given $d:d_1,d_2,\ldots,d_\ell$ with $\dn{d}\in \cIGr(\de)$.  
Let $\phi(\dn{d})=\up{e}$, where $e:e_1,e_2,\ldots,e_\ell$.

We first show that $\phi$ is well defined in that 
$\#\dn{d}=\#\up{e}$ and
$\up{e}\in \cUGr(\de)$.  
The first statement is clear since \ref{P1} does not change cardinalities, and every entry increased by one in \ref{P2} is offset by an entry decreased by one.  For the second statement, we need to check \ref{U1}--\ref{U3}.  The truth of \ref{U1} follows from the fact that $d$ satisfies \ref{I1} unless $d_i=\de_i+1$ and $d_i$ is increased in step \ref{P2}.  But if $i=1$ then this contradicts \ref{I3}, and if $i>1$ then this contradicts \ref{I2}, since $d_i$ was not the first nonzero entry in its block.
If \ref{U2} is violated, then $e_i=\de_i+1$  and $e_{i+1}=0$.  So $d_i$ must have been the last entry of some $B'$.  If $\#B'=1$ then $d_{i-1}=0$. But then $e_i=d_i\le \de_i$, because if we had $d_i=\de_i+1$ then $\dn{d}$ would not be an ideal since it violates \ref{I2}.
On the other hand, if $\#B'>1$ then by~\ref{P2} we have $e_i=d_i-1\le\de_i$, which is another contradiction.
Thus  $\up{e}$ is a filter.
 Finally, we must verify \ref{U3}. For the first condition suppose, towards a contradiction, that $e_1=1$, and let $B'$ be the initial factor of the block $B$ containing $d_1\ge1$.  If $e_1=d_1$ then by~\ref{P2} we must have $\#B'=1$.  But then $B'$ would have been included in the trailing ones of $B$ and moved to the right in~\ref{P1}.
 The other possibility is $e_1=d_1+1\ge2$, again a contradiction.  Thus the first condition holds.
 To prove that the second condition is true, assume the opposite which is that $e_\ell=\de_\ell+1$. Clearly $e_\ell\ge2$.  It follows that $d_\ell$ must have been part of a block $B$ with no trailing ones so that $B'=B$.  If $\#B'=1$ then $d_\ell=e_\ell=\de_\ell+1$.  By~\ref{P2}, this forces $d_{\ell-1}\neq 0$.  But then $d_\ell$ was not the only element in $B'$ which is impossible.
If $\#B'\ge2$ then, by~\ref{P2} again, $e_\ell=d_\ell-1\le \de_\ell$.  This final contradiction finishes the proof that~\ref{U3} holds and that $\phi$ is well defined.

 To show that $\phi$ is bijective, we construct $\phi^{-1}:\cUGr(\de)\ra \cIGr(\de)$.  If $\up{e}\in \cUGr(\de)$ then define
 $$
 \phi^{-1}(\up{e}) = \phi(\up{e}^R)^R,
 $$
 where, on the right-hand side, the map being applied is
 $\phi:\cIGr(\de^R)\ra\cUGr(\de^R)$.
 Because reversal is an involution, showing that $\phi^{-1}$ is well defined is equivalent to showing that $\up{e}$ is a restricted filter if and only if $\up{e}^R$ is a restricted ideal.  But this follows immediately by comparing \ref{I1}--\ref{I3} with \ref{U1}--\ref{U3}.
 
 To show that $\phi^{-1}$ is indeed the inverse of $\phi$, we claim that the factors of ones moved by $\phi$ are the same as those moved by $\phi^{-1}$.  We will only show that if a factor is moved by $\phi$ then it is moved by $\phi^{-1}$, as the reverse implication is similar.
 Let $T$ be a factor of trailing ones in $\dn{d}$. After $T$ moves when applying $\phi$, it either becomes a block itself or merges with $B'$ where $B$ is the block which was to its right.  In the first case, $T$ is clearly a block of ones in $\up{e}=\phi(\dn{d})$ and so also in $\up{e}^R$.  Thus it will be moved when computing $\phi^{-1}(\up{e})$.  In the second case, it suffices to show that the leftmost entry $d_i$ of $B$ becomes $e_i\ge2$ in $\up{e}$, since then $T$ becomes a factor of trailing ones in $\up{e}^R$.  If $\#B=1$ then $e_i=d_i\ge2$ since, otherwise, $d_i=1$ would have been one of the trailing ones of the original block and moved to the right.  On the other hand, if $\#B\ge2$, then by \ref{P2} we have $e_i=d_i+1\ge2$, which is again what we wished to show and completes the second case of the claim. 
 
 Because of the claim, $\phi^{-1}$ acts as a step-by-step inverse of $\phi$.  Indeed, moving factors right in $\dn{d}$ corresponds to moving the same factors left in $\up{e}$.  And this is equivalent to moving them right in $\up{e}^R$ by applying $\phi$, while the final reversal brings the factor back to its original position.  Also, what \ref{P2} does to the two ends of the remains $B'$ of a block $B$ are inverses of each other.  This shows that $B'$ will also be restored to itself by $\phi^{-1}$, so that this map does indeed undo what was done by $\phi$.  This completes the proof.
 \end{proof}

\subsection{Bijection $\Phi$ for fences}
\label{bpf}

Let $\be=(\be_1,\be_2,\dots,\be_s)\comp n-1$ with $s$ odd.
Write $s=2\ell+1$, where $\ell\ge0$. 
Our algorithm will be simplest to state using somewhat different parameters for the corresponding fence $F=F(\be)$.  
These constants first appeared in the work of Elizalde, Plante, Roby and Sagan~\cite{EPRS:rf} concerning rowmotion on fences.
Call the elements of $F$ which appear on two segments {\em shared} and all other elements {\em unshared}.
It will be convenient to use different notation and conventions for ascending and descending segments. 
Let the ascending segments of $F$ be $A_1,A_2,\dots,A_{\ell+1}$ from left to right, and similarly let $D_1,D_2,\dots,D_\ell$ be the descending segments. 
Let
\begin{equation}
\label{de_i}
 \de_i = 1+(\text{the number of unshared elements on $D_{2i}$})   
\end{equation}
for $1\le i\le \ell$.
Thus $\de_i=\be_{2i}$.  Similarly, let 
\begin{equation}
\label{al_i}
   \al_i = 1+(\text{the number of unshared elements on $A_{2i-1}$}) 
\end{equation}
for $1\le i\le \ell+1$.
It follows that  $\al_i=\be_{2i-1}$ for $2\le i\le \ell$, $\al_1=\be_1+1$, and $\al_{\ell+1}=\be_s+1$.
For $i\in[\ell+1]$, denote by $\tA_i$ the chain consisting of the unshared elements on segment $A_i$. Note that
 $\#\tA_i=\al_i-1$ and $\#D_i=\de_i+1$, and that each element from $F$ appears in exactly one of the $\tA_i$ or $D_i$.

We encode ideals $I$ of $F(\be)$ by pairs of sequences $a:a_1,a_2,\dots,a_{\ell+1}$ and $d:d_1,d_2,\dots,d_\ell$,
where $a_i=\#(I\cap \tA_i)$ and $d_i=\#(I\cap D_i)$ for all $i$. It is sometimes convenient to visualize these sequences as placed one above the other, with entries interlaced, i.e.,
$$
\Dn{\barr{c}a\\ d\earr}=\Dn{\begin{array}{cccccccccc} a_1 && a_2 &&\cdots&&a_{\ell}&& a_{\ell+1}\\
& d_1 && d_2 &&\cdots&& d_\ell &
\end{array}}.
$$
Similarly, we encode filters $U$ of $F(\be)$ by pairs of sequences $b:b_1,b_2,\dots,b_{\ell+1}$ and $e:e_1,e_2,\dots,e_\ell$,
where $b_i=\#(U\cap \tA_i)$ and $e_i=\#(U\cap D_i)$ for all $i$, and we write
$$
\Up{\barr{c} b \\ e\earr}=\Up{\begin{array}{cccccccccc} b_1 && b_2 &&\cdots&&b_{\ell}&& b_{\ell+1}\\
& e_1 && e_2 &&\cdots&& e_\ell &
\end{array}}.
$$

A pair of sequences $\dn{\pair{a}{d}}$ as above encodes an ideal of $F(\be)$ if and only if the following conditions hold:
\begin{enumerate}[label={IF\arabic*}]
    \item \label{IF1} for $i\in[\ell+1]$ we have $0\le a_i\le\al_i-1$,
    \item \label{IF2} for $i\in[\ell]$ we have $0\le d_i\le \de_i+1$,
    \item \label{IF3} for $i\in[\ell]$: if $d_i=\de_i+1$ then $a_{i}=\al_{i}-1$,  and if $i>1$ then $d_{i-1}>0$ as well,
    \item \label{IF4}  for $i\in[\ell]$: if $a_{i+1}>0$ then $d_i>0$.
\end{enumerate}
Note that the size of the ideal is $\sum_i a_i+\sum_i d_i$.

To obtain the conditions for $\up{\pair{b}{e}}$ to encode a filter of $F(\be)$, note that this happens if and only if  
\begin{equation}
\label{upR}
    \up{\pair{b}{e}}^R\stackrel{\text{def}}{=}\dn{\pair{b^R}{e^R}}
\end{equation}
encodes an ideal of $F(\be^R)$.  
Similarly define 
\begin{equation}
\label{dnR}
\dn{\pair{a}{d}}^R=\up{\pair{a^R}{d^R}}.
\end{equation}
The following is equivalent to $\up{\pair{b}{e}}$ being a filter of $F(\be)$:
\begin{enumerate}[label={UF\arabic*}]
    \item \label{UF1} for $i\in[\ell+1]$ we have $0\le b_i\le\al_i-1$,
    \item \label{UF2} for $i\in[\ell]$ we have $0\le e_i\le \de_i+1$,
    \item \label{UF3} for $i\in[\ell]$: if $e_i=\de_i+1$ then $b_{i+1}=\al_{i+1}-1$,  and if $i<\ell$ then $e_{i+1}>0$ as well,
    \item \label{UF4}  for $i\in[\ell]$: if $b_i>0$ then $e_i>0$.
\end{enumerate}

Next, we define a bijection $\Phi:\cI_k(F)\to\cU_k(F)$, where 
$$
k\le\min\{\be_1,\be_s\}=\min\{\al_1,\al_{\ell+1}\}-1.
$$
Given an ideal of $\cI_k(F)$ encoded by a pair of sequences $\dn{\pair{a}{d}}$, we apply the following steps, where we use $x:=y$ to mean that $x$ is to be assigned the value  $y$.
\begin{enumerate}[label={PH\arabic*}]
    \item \label{PH1} For every $i\in[\ell]$ such that $d_i=1$ and $a_{i+1}<\al_{i+1}-1$, let $d_i:=0$ and $a_{i+1}:=a_{i+1}+1$.
    \item \label{PH2} Decompose  $d:d_1,d_2,\dots,d_\ell$ into factors by splitting between $d_{i-1}$ and $d_i$ for each $i\in[2,\ell]$ such that $a_{i}<\al_{i}-1$. Apply $\phi$ (defined by \ref{P1}--\ref{P2}) to each factor to obtain a sequence $e$.  Let $b:=a$.
    \item \label{PH3} For every $i\in[\ell]$ such that $e_i=0$ and $b_i>0$, let  $e_i:=1$ and $b_i:=b_i-1$.
\end{enumerate}
Define $$\Phi(\dn{\pair{a}{d}})=\up{\pair{b}{e}}.$$

For example, let $F=F(6,2,1,2,3,1,6)$ be the fence in Figure~\ref{Phi-example}, which has $\al:7,1,3,7$ and $\de:2,2,1$. Label its elements $x_1,x_2,\dots,x_{22}$ from left to right, and consider the ideal $I=\{x_9,x_{10},x_{11},x_{12},x_{13},x_{16}\}\in\cI(F)$, which is encoded by
$$\Dn{\barr{c}a\\ d\earr}=\Dn{\begin{array}{cccccccc} 0 && 0 && 1 &&0\\
& 1 && 3 && 1 &
\end{array}}.
$$
In the top of Figure~\ref{Phi-example} the elements of $I$ are circled. Applying \ref{PH1} yields
$$\begin{array}{cccccccc} 0 && 0 && 1 &&1\\
& 1 && 3 && 0 &
\end{array}.
$$
In step \ref{PH2}, the sequence $1,3,0$ is split into two factors $1,3$ and $0$. Applying $\phi$ to each one we get
$$\begin{array}{cccccccc} 0 && 0 && 1 &&1\\
& 2 && 2 && 0 &
\end{array}.
$$
Finally, applying \ref{PH3} yields
$$\Up{\barr{c} b \\ e\earr}=\Up{\begin{array}{cccccccc} 0 && 0 && 0 &&1\\
& 2 && 2 && 1 &
\end{array}},
$$
which encodes the filter $U=\{x_7,x_8,x_{10},x_{11},x_{15},x_{21}\}\in\cI(F)$, depicted in the bottom of Figure~\ref{Phi-example}.

\begin{figure}[htb]\centering
\begin{tikzpicture}[scale=.6]
\draw(-1,3) node{$I=$};
\draw (0,0)--(6,6)--(8,4)--(9,5)--(11,3)--(14,6)--(15,5)--(21,11);
\foreach \position in {(0,0), (1,1), (2,2), (3,3), (4,4), (5,5), (6,6), (7,5), (8,4), (9,5), (10,4), (11,3), (12,4), (13,5), (14,6), (15,5), (16,6), (17,7), (18,8), (19,9), (20,10), (21,11)}
   \fill\position circle(.1);
\foreach \position in {(8,4), (9,5), (10,4), (11,3), (12,4), (15,5)}
   \draw[red,thick]\position circle(.2);
\draw(0,.7) node{$x_1$};
\draw(1,1.7) node{$x_2$};
\draw(2,2.7) node{$x_3$};
\draw(3,3.7) node{$x_4$};
\draw(4,4.7) node{$x_5$};
\draw(5,5.7) node{$x_6$};
\draw(6,6.7) node{$x_7$};
\draw(7,5.7) node{$x_8$};
\draw(8,4.7) node{$x_9$};
\draw(9,5.7) node{$x_{10}$};
\draw(10,4.7) node{$x_{11}$};
\draw(11,3.7) node{$x_{12}$};
\draw(12,4.7) node{$x_{13}$};
\draw(13,5.7) node{$x_{14}$};
\draw(14,6.7) node{$x_{15}$};
\draw(15,5.7) node{$x_{16}$};
\draw(16,6.7) node{$x_{17}$};
\draw(17,7.7) node{$x_{18}$};
\draw(18,8.7) node{$x_{19}$};
\draw(19,9.7) node{$x_{20}$};
\draw(20,10.7) node{$x_{21}$};
\draw(21,11.7) node{$x_{22}$};
\draw[|->] (10,1)-- node[right]{$\Phi$} (10,0);
\begin{scope}[shift={(0,-7)}]
\draw(-1,3) node{$U=$};
\draw (0,0)--(6,6)--(8,4)--(9,5)--(11,3)--(14,6)--(15,5)--(21,11);
\foreach \position in {(0,0), (1,1), (2,2), (3,3), (4,4), (5,5), (6,6), (7,5), (8,4), (9,5), (10,4), (11,3), (12,4), (13,5), (14,6), (15,5), (16,6), (17,7), (18,8), (19,9), (20,10), (21,11)}
   \fill\position circle(.1);
\foreach \position in {(6,6), (7,5), (9,5), (10,4), (14,6), (21,11)}
   \draw[red,thick]\position circle(.2);
\draw(0,.7) node{$x_1$};
\draw(1,1.7) node{$x_2$};
\draw(2,2.7) node{$x_3$};
\draw(3,3.7) node{$x_4$};
\draw(4,4.7) node{$x_5$};
\draw(5,5.7) node{$x_6$};
\draw(6,6.7) node{$x_7$};
\draw(7,5.7) node{$x_8$};
\draw(8,4.7) node{$x_9$};
\draw(9,5.7) node{$x_{10}$};
\draw(10,4.7) node{$x_{11}$};
\draw(11,3.7) node{$x_{12}$};
\draw(12,4.7) node{$x_{13}$};
\draw(13,5.7) node{$x_{14}$};
\draw(14,6.7) node{$x_{15}$};
\draw(15,5.7) node{$x_{16}$};
\draw(16,6.7) node{$x_{17}$};
\draw(17,7.7) node{$x_{18}$};
\draw(18,8.7) node{$x_{19}$};
\draw(19,9.7) node{$x_{20}$};
\draw(20,10.7) node{$x_{21}$};
\draw(21,11.7) node{$x_{22}$};
\end{scope}
\end{tikzpicture}
\capt{Computing $\Phi(\{x_9,x_{10},x_{11},x_{12},x_{13},x_{16}\})$ 
in $F(6,2,1,2,3,1,6)$. \label{Phi-example}}
\end{figure}

We now prove the main theorem of this section.
\begin{thm} \label{thm:Phi}
Let $\be=(\be_1,\be_2,\ldots,\be_s)$ where $s=2\ell+1$ and 
\begin{equation}
\label{kle}
    k\le\min\{\be_1,\be_s\}.
\end{equation}
The map $\Phi:\cI_k(\be)\ra \cU_k(\be)$  defined by~\ref{PH1}--\ref{PH3} is a bijection.
\end{thm}
\begin{proof}
We maintain the notation established in the lead up to this theorem.
To show that $\Phi$ is well defined, we need to first demonstrate that $\phi$ can be applied to the factors determined by \ref{PH2} in that they satisfy \ref{I1}--\ref{I3}.  The first two conditions follow directly from the fact that $I=\dn{\pair{a}{d}}$ is an ideal.  For \ref{I3}, we begin with $d_1$ in the first factor and assume, towards a contradiction, that $d_1=\de_1+1$.  But then \ref{IF3} forces $a_1=\al_1-1$. So
$$
k=\#I\ge a_1+d_1 \ge\al_1 =\be_1+1,
$$
which contradicts~\eqref{kle}.  Now consider $d_\ell$ in the last factor and suppose, again towards a contradiction, that $d_\ell=1$ when $\phi$ is about to be applied.
Note that we must also have $a_{\ell+1}<\al_{\ell+1}-1$, since otherwise we would again contradict~\eqref{kle} similarly to our first case.  But under these conditions, \ref{PH1} would have set $d_\ell$ to $0$, which is again a contradiction. To finish verifying \ref{I3}, we must consider the splits between $d_{i-1}$ and $d_i$ for $2\le i\le\ell$, which occur when $a_i<\al_i-1$ in \ref{PH2}.
If $d_{i-1}=1$ then we must have had $a_i<\al_i-1$ to start with, since $a_i$ can only increase in value, so \ref{PH1} would have again set $d_{i-1}$ to $0$.  If $d_i=\de_1+1$ then \ref{IF3} forces $a_i=\al_i-1$, which contradicts the assumption in \ref{PH2}.  So in all cases $\phi$ can be applied.

That $\Phi$ preserves cardinality follows from the fact that $\phi$ does and that the assignments in steps \ref{PH1} and \ref{PH3} keep the sum of the sequences equal.  So to finish the proof that $\Phi$ is well defined we need to show that
$\Phi(\dn{\pair{a}{d}})=\up{\pair{b}{e}}$ satisfies \ref{UF1}--\ref{UF4}.
The first two items follow by the equalities and bounds imposed in \ref{PH1} and \ref{PH3} before reassignment, and from the fact that the image of $\phi$ satisfies \ref{U1}.  To check \ref{UF3}, suppose $e_i=\de_i+1$.  The ``and if" clause is  true because applying $\phi$ gives a sequence satisfying \ref{U2}.  For the first clause we will see that having $b_{i+1}<\al_{i+1}-1$ leads to a contradiction.  
Note that the value of $b_{i+1}$ could not have been lowered in \ref{PH3} since $e_i=\de_i+1\neq 0$.  So we have $a_{i+1}=b_{i+1}<\al_{i+1}-1$, and the condition in \ref{PH2} forces $e_i$ to be the end of a factor.  But since $\phi$ maps to restricted filters, we have that $e_i\le \de_i$ by \ref{U3}, which is the desired contradiction.  Finally, we tackle \ref{UF4} by contradiction again, assuming $b_i>0$ and $e_i=0$.  If this had been the case, then $e_i$ would have been reassigned to be $1$ in \ref{PH3}.  This completes the verification that $\Phi$ is well defined.

As with $\phi$, we define $\Phi^{-1}$ to be
\begin{equation}\label{eq:Phiinv}
\Phi^{-1}(\up{\pair{b}{e}}) = \Phi(\up{\pair{b}{e}}^R)^R.
\end{equation}
As in the demonstration of Theorem~\ref{thm:phi}, the proof that $\Phi^{-1}$ is well defined follows from the fact that $\Phi$ is.

We first prove that $\Phi^{-1}\circ\Phi$ is the identity.
We first need to show that if $\Phi(\dn{\pair{a}{d}}) =\up{\pair{b}{e}}$ then $d$ gets broken into factors when applying $\Phi$ at the same indices as $e^R$ when applying $\Phi^{-1}$.  We will show that every break point of $d$ becomes a break point of $e^R$, again leaving the reverse implication to the reader.  If there was a break between $d_{i-1}$ and $d_i$ in applying $\Phi$, then we must have $a_i<\al_i-1$ in step \ref{PH2}. After applying \ref{PH3} we have $b_i\le a_i<\al_i-1$.  
Next \ref{PH1} is applied to $\up{\pair{b}{e}}^R$ as the first step of $\Phi^{-1}$.  If $b_i$ does not change at this step, then \ref{PH2} will still split $e^R$ between $e_i$ and $e_{i-1}$ because of the previous inequality.  If $b_i$ does increase during \ref{PH1}, then it must have been because $e_i=1$ at this stage.  But $d_i$ was first in its factor before applying $\phi$,
and so, by \ref{U3}, we had $e_i\neq 1$ after \ref{PH2} was applied as part of $\Phi$.  So the only way to have $e_i=1$ at the end of \ref{PH3} is if we also decreased $b_i$ by one in that step.  In this case $b_i<a_i<\al_i-1$, which makes $b_i<\al_i-1$ after adding one in \ref{PH1}.  So \ref{PH2} will still break at the same spot.

It is now easy to see that $\Phi^{-1}$ will act as a step-by-step inverse for $\Phi$.  Indeed, applying \ref{PH1} for $\Phi^{-1}$ undoes what \ref{PH3} did for $\Phi$.  By what we proved in the previous paragraph and the definition of $\phi^{-1}$, the steps \ref{PH2} for $\Phi$ and $\Phi^{-1}$ cancel each other out.  And finally step \ref{PH3} for $\Phi^{-1}$ cancels out \ref{PH1} in $\Phi$.  

To complete the proof,  we show that $\Phi^{-1}\circ\Phi$ is the identity map.  This follows from equation~\eqref{eq:Phiinv} and the fact that $\Phi^{-1}\circ\Phi$ is the identity, since
$$
\Phi(\Phi^{-1}(\up{\pair{b}{e}})) = \Phi(\Phi(\up{\pair{b}{e}}^R)^R)
=\Phi^{-1}(\Phi(\up{\pair{b}{e}}^R))^R
=(\up{\pair{b}{e}}^R)^R=\up{\pair{b}{e}}
$$
as desired.
\end{proof}


\section{Proof of symmetry for circular fences}
\label{psc}

We will now show how slight modifications of $\phi$ and $\Phi$ can be used to give a bijective proof of Theorem~\ref{FbSym}.  We use the notation
$$
\cIb(\be) = \{I \mid \text{$I$ is a lower order ideal of $\Fb(\be)$}\}
$$
and
$$
\cUb(\be) = \{U \mid \text{$U$ is an upper order ideal of $\Fb(\be)$}\}.
$$
Our goal is to  construct a cardinality-preserving bijection $\Phib:\cIb(\be)\ra\cUb(\be)$.
As before, we start with the case where ascending segments have length one.

\subsection{Bijection $\phib$ for narrow circular fences}

We call a circular fence $\Fb(\be)$ {\em narrow} if its composition has the form $\be=(1,\de_1,1,\de_2,\dots,1,\de_\ell)$.  Let $D_i$ be the descending segment of length $\de_i$.  Any $I\in\cIb(\be)$ can be expressed as $I=\dn{d_1,d_2,\ldots,d_\ell}$, 
where $d_i=\#(I\cap D_i)$ for $i\in[\ell]$, satisfying the following conditions:
\begin{enumerate}[label={ICN\arabic*}]
    \item \label{ICN1} (existence) for $i\in[\ell]$ we have $0\le d_i\le\de_i+1$,
    \item \label{ICN2} (ideal) for $i\in[\ell]$: if $d_i=\de_i+1$ then $d_{i-1}>0$, where subscripts are taken modulo $\ell$.
\end{enumerate}
Similarly, the conditions for filters 
$U=\up{e_1,e_2,\ldots,e_\ell}$ of $\Fb(\be)$ are as follows:
\begin{enumerate}[label={UCN\arabic*}]
    \item \label{UCN1} (existence) for $i\in[\ell]$ we have $0\le e_i\le\de_i+1$,
    \item \label{UCN2} (filter) for $i\in[\ell]$: if $e_i=\de_i+1$ then $e_{i+1}>0$, where subscripts are taken modulo $\ell$.
\end{enumerate}

To define $\phib$, it will be useful to define a {\em circular sequence}
$\spn{d}: \spn{d_1,d_2,\ldots,d_\ell}$, which is obtained from the linear sequence $d:d_1,d_2,\ldots,d_\ell$  by considering $d_\ell$ as followed by $d_1$.  Equivalently, the subscripts in a circular sequence are to be treated modulo $\ell$ and this will be our convention in all definitions pertaining to circular sequences.  
  Note our calling ordinary sequences {\em linear} to distinguish them from the circular case.

A {\em factor} of  $\spn{d}$ is a subsequence of the form $d_i,d_{i+1},\ldots,d_j$.
Note that this is a linear sequence even though it may wrap around to the beginning of $d$.  Call $d$ {\em positive} if all its elements are positive.
If $d$ is not positive (and so has at least one zero) then a {\em block} $B$ of $\spn{d}$ is a maximal factor of positive elements.  For example, the circular sequence
$$
\spn{d} = \spn{7,1,1,0,5,1,0,0,3}
$$
has blocks $5,1$ and $3,7,1,1$.  Now the {\em trailing ones} of a block are defined exactly as in the linear case. Conveniently, for circular sequences  every factor of trailing ones is followed by a zero, which is why we do not need the notion of restriction for ideals in circular fences.  In our example, block $5,1$ has one trailing one and block $3,7,1,1$ has two.

Now suppose we are given $I=\dn{d_1,d_2,\ldots,d_\ell}\in\cIb(\be)$ with $\Fb(\be)$ narrow. If $d:d_1,d_2,\ldots,d_\ell$ is not positive, then we define 
$\phib(I)$ by applying~\ref{P1} and~\ref{P2} for $\phi$ to $\spn{d_1,d_2,\ldots,d_\ell}$.
Note that this is well defined since factors of a circular permutation are linear.  Returning to our example, we have
$$
I=\dn{7,1,1,0,5,1,0,0,3} \stackrel{\text{\ref{P1}}}{\mapsto}
\spn{7,0,1,1,5,0,1,0,3}\stackrel{\text{\ref{P2}}}{\mapsto}
\up{6,0,1,1,5,0,1,4}=\phib(I).
$$
If $d$ is positive, then we let 
$$
\phib\dn{d_1,d_2,\ldots,d_\ell} =\up{d_1,d_2,\ldots,d_\ell}.
$$

The proof of the next result is very similar to that of Theorem~\ref{thm:phi}, so the demonstration is omitted.
\begin{thm}
Let $\be=(1,\de_1,1,\de_2,\dots,1,\de_\ell)$. The map $\phib:\cIb(\be)\ra \cUb(\be)$  defined above is a cardinality-preserving bijection.\hqed
\end{thm}

\subsection{Bijection $\Phib$ for circular fences}

Now consider an arbitrary circular fence $\Fb=\Fb(\be)$, where $\be=(\be_1,\be_2,\ldots,\be_{2\ell})\comp n$.  We 
again use $A_i$ and $D_i$ to denote the corresponding ascending and descending segments, noting that now there are only $\ell$ ascending segments.  Define $\de_i$ and $\al_i$ using equations~\eqref{de_i} and~\eqref{al_i}, respectively.  We now have $\de_i=\be_{2i}$ and $\al_i=\be_{2i-1}$ for all $i$ (unlike the linear case, there are no exceptions).

Given $I\in\cIb(\be)$, we continue to let $a_i=\#(I\cap \tA_i)$ and 
$d_i=\#(I\cap D_i)$, where $\tA_i$ also retains its previous meaning.  The notation for $I$ will be
\begin{equation}
\label{DnCir}
    \Dn{\barr{c}a\\ d\earr}=\Dn{\begin{array}{cccccccccc} a_1 && a_2 &&\cdots&&a_{\ell}&&a_1\\
    & d_1 && d_2 &&\cdots&& d_\ell&
\end{array}}.
\end{equation}
Note the repetition of $a_1$ in the top line, which will make our future definitions simpler.
The encoding for filters is changed mutatis mutandis.

We can now easily write down the conditions for being an ideal of a circular fence in terms of the $a_i$ and $d_i$ (all subscripts are modulo $\ell$):
\begin{enumerate}[label={IC\arabic*}]
    \item \label{IC1} for $i\in[\ell]$ we have $0\le a_i\le\al_i-1$,
    \item \label{IC2} for $i\in[\ell]$ we have $0\le d_i\le \de_i+1$,
    \item \label{IC3} for $i\in[\ell]$: if $d_i=\de_i+1$ then $a_{i}=\al_{i}-1$ and $d_{i-1}>0$,
    \item \label{IC4}  for $i\in[\ell]$: if $a_{i}>0$, then $d_{i-1}>0$.
\end{enumerate}
Similarly, $\up{\pair{b}{e}}$ being a filter is equivalent to the following conditions:
\begin{enumerate}[label={UC\arabic*}]
    \item \label{UC1} for $i\in[\ell]$ we have $0\le b_i\le\al_i-1$,
    \item \label{UC2} for $i\in[\ell]$ we have $0\le e_i\le \de_i+1$,
    \item \label{UC3} for $i\in[\ell]$: if $e_i=\de_i+1$ then $b_{i+1}=\al_{i+1}-1$ and $e_{i+1}>0$,
    \item \label{UC4}  for $i\in[\ell]$: if $b_{i}>0$, then $e_i>0$.
\end{enumerate}

We now modify~\ref{PH1}--\ref{PH3} for the circular case.
Given $\dn{\pair{a}{d}}$ as in~\eqref{DnCir}, perform the following operations.  In all steps, the indices are taken modulo $\ell$.
\begin{enumerate}[label={PHC\arabic*}]
    \item \label{PHC1} For every $i\in[\ell]$ such that $d_i=1$ and $a_{i+1}<\al_{i+1}-1$, let $d_i:=0$ and $a_{i+1}:=a_{i+1}+1$.
    \item \label{PHC2} If there exists some index $i\in[\ell]$ with $a_{i}<\al_{i}-1$, then split $\spn{d}$ into  factors between $d_{i-1}$ and $d_i$ for each such $i$ and apply $\phi$ to each factor. If no such $i$ exists, then compute $\phib(\dn{d})$.  In both cases, let $e$ be the resulting sequence. Let $b:=a$.
    \item \label{PHC3} For every $i\in[\ell]$ such that $e_i=0$ and $b_i>0$, let  $e_i:=1$ and $b_i:=b_i-1$.
\end{enumerate}
Define 
$$
\Phib(\dn{\pair{a}{d}})=\up{\pair{b}{e}}.
$$

Let us look at two examples which will illustrate the two cases in step~\ref{PHC2}.  First consider the circular fence
$\Fb(\be)$ where $\be=(2,1,2,3,1,2,2,1)$, so
$$
\al:2,2,1,2 \qmq{and} \de:1,3,2,1,
$$
as illustrated in Figure~\ref{EqualCase}.
Let $I=\{x_1,x_2,x_3,x_4,x_5,x_9,x_{12}\}\in\cIb(\be)$, which is encoded by
$$\Dn{\barr{c}a\\ d\earr}=\Dn{\begin{array}{cccccccccc} 1 && 1 && 0 && 0 && 1\\
& 2 && 1 && 1 && 1&
\end{array}}.
$$
This ideal is indicated by the circled nodes in the top poset in Figure~\ref{EqualCase}.
Applying~\ref{PHC1} yields
$$\begin{array}{cccccccccc} 1 && 1 && 0 && 1 && 1\\
& 2 && 1 && 0 && 1&
\end{array}.
$$
In step~\ref{PHC2}, since $a_i=\al_i-1$ for all $i$, we simply apply $\phib$ to the sequence $\dn{d}=\dn{2,1,0,1}$, which gives $\up{e}=\up{1,0,1,2}$. Finally, applying~\ref{PHC3} to
$$\begin{array}{cccccccccc} 1 && 1 && 0 && 1 && 1\\
& 1 && 0 && 1 && 2&
\end{array}
$$
yields
$$\Up{\barr{c}b\\ e\earr}=\Up{\begin{array}{cccccccccc} 1 && 0 && 0 && 1 && 1\\
& 1 && 1 && 1 && 2&
\end{array}},
$$
which encodes the filter $U=\{x_1,x_2,x_3,x_6,x_{10},x_{13},x_{14}\}\in\cUb(\be)$,
as illustrated in the bottom poset in Figure~\ref{EqualCase}.

\begin{figure} \centering
\begin{tikzpicture}[scale=.7]
\draw(-1,4) node{$I=$};
\draw (10,0)--(0,4)--(1,5)--(2,4)--(4,6)--(7,3)--(8,4)--(10,2)--(12,4)--(10,0);
\foreach \position in {(0,4), (1,5), (2,4), (3,5), (7,3), (10,2), (10,0)}
   \draw[red,thick]\position circle(.2);
\foreach \position in {(4,6), (5,5), (6,4), (8,4), (9,3), (11,3), (12,4),(0,4), (1,5), (2,4), (3,5), (7,3), (10,2), (10,0)}
   \fill\position circle(.1);
\draw(10,-.5) node{$x_1$};
\draw(0,4.5) node{$x_2$};
\draw(1,5.5) node{$x_3$};
\draw(2,4.5) node{$x_4$};
\draw(3,5.5) node{$x_5$};
\draw(4,6.5) node{$x_6$};
\draw(5,5.5) node{$x_7$};
\draw(6,4.5) node{$x_8$};
\draw(7,3.5) node{$x_9$};
\draw(8,4.5) node{$x_{10}$};
\draw(9,3.5) node{$x_{11}$};
\draw(10,2.5) node{$x_{12}$};
\draw(11,3.5) node{$x_{13}$};
\draw(12,4.5) node{$x_{14}$};
\draw[|->] (6,.2)-- node[right]{$\Phib$} (6,-.5);

\begin{scope}[shift={(0,-7)}]
\draw(-1,4) node{$U=$};
\draw (10,0)--(0,4)--(1,5)--(2,4)--(4,6)--(7,3)--(8,4)--(10,2)--(12,4)--(10,0);
\foreach \position in {(0,4), (1,5), (4,6), (8,4),  (10,0), (11,3), (12,4)}
   \draw[red,thick]\position circle(.2);
\foreach \position in {(4,6), (5,5), (6,4), (8,4), (9,3), (11,3), (12,4),(0,4), (1,5), (2,4), (3,5), (7,3), (10,2), (10,0)}
   \fill\position circle(.1);
\draw(10,-.5) node{$x_1$};
\draw(0,4.5) node{$x_2$};
\draw(1,5.5) node{$x_3$};
\draw(2,4.5) node{$x_4$};
\draw(3,5.5) node{$x_5$};
\draw(4,6.5) node{$x_6$};
\draw(5,5.5) node{$x_7$};
\draw(6,4.5) node{$x_8$};
\draw(7,3.5) node{$x_9$};
\draw(8,4.5) node{$x_{10}$};
\draw(9,3.5) node{$x_{11}$};
\draw(10,2.5) node{$x_{12}$};
\draw(11,3.5) node{$x_{13}$};
\draw(12,4.5) node{$x_{14}$};
\end{scope}
\end{tikzpicture}
\capt{Computing $\Phib(\{x_1,x_2,x_3,x_4,x_5,x_9,x_{12}\})$ in $\Fb(2,1,2,3,1,2,2,1)$ \label{EqualCase}}
\end{figure}

If instead we apply $\Phib$ to the ideal  $I=\{x_1,x_2,x_3,x_4,x_9,x_{12}\}\in\cIb(\be)$, which is encoded by
$$\Dn{\barr{c}a\\ d\earr}=\Dn{\begin{array}{cccccccccc} 1 && 0 && 0 && 0 && 1\\
& 2 && 1 && 1 && 1&
\end{array}},
$$
step~\ref{PHC1} yields
$$\begin{array}{cccccccccc} 1 && 0 && 0 && 1 && 1\\
& 2 && 1 && 0 && 1&
\end{array}.
$$
Now, in step~\ref{PHC2}, we have $a_2=0<1=\al_2-1$, so we split $\spn{d}$ between $d_1=2$ and $d_2=1$. Applying $\phi$ to the resulting linear factor $1,0,1,2$, we obtain the sequence $0,1,2,1$, and so $e:1,0,1,2$. Finally, applying~\ref{PHC3} to
$$\begin{array}{cccccccccc} 1 && 0 && 0 && 1 && 1\\
& 1 && 0 && 1 && 2&
\end{array}
$$
does not produce any change, and so
$$\Up{\barr{c}b\\ e\earr}=\Up{\begin{array}{cccccccccc} 1 && 0 && 0 && 1 && 1\\
& 1 && 0 && 1 && 2&
\end{array}},
$$
which encodes the filter $U=\{x_1,x_2,x_3,x_{10},x_{13},x_{14}\}\in\cUb(\be)$.

To prove that $\Phib$ is bijective, we will use the definition of reversal for ideals given by~\eqref{dnR}, remembering that for circular fences $a_1$ appears twice in $a$.  So,
$$
\Dn{\begin{array}{cccccccccc} a_1 && a_2 &&\cdots&&a_{\ell}&& a_1\\
& d_1 && d_2 &&\cdots&& d_\ell &
\end{array}}^R =
\Up{\begin{array}{cccccccccc} a_1 && a_\ell &&\cdots&&a_2&& a_1\\
& d_{\ell} && d_{\ell-1} &&\cdots&& d_1 &
\end{array}}.
$$
Similarly, reversal for filters is given by~\eqref{upR}.
\begin{thm}
Let $\be=(\be_1,\be_2,\ldots,\be_{2\ell})$. The map $\Phib:\cIb(\be)\ra \cUb(\be)$  defined by~\ref{PHC1}--\ref{PHC3} is a cardinality-preserving bijection.
\end{thm}
\begin{proof}
We will use the notation we have established above. If there is an index $i$ in step \ref{PHC2} with $a_i<\al_i-1$, then this map is very similar to $\Phi$.  The proof in this case essentially follows the lines of that of Theorem~\ref{thm:Phi}, and so we omit the details.

Assume now that, in step \ref{PHC2}, we have $a_i=\al_i-1$ for all $i$. There are two possibilities depending on whether $d$ is positive or not.  First consider what happens if $d$ is positive.  In this case 
$\phib(\dn{d})=\up{d}$, so that in step~\ref{PHC2} we have $e:=d$.  Since $d$ does not change, \ref{PHC3} will undo what was done in~\ref{PHC1}, so that $b=a$.
Thus, in this case $\Phib$ is the identity map at the level of encodings.  It is now easy to check that this map is well defined, and trivial that it is a bijection.

Now suppose that $d$ is not positive.  Clearly $\Phib$ preserves cardinality, because so does $\phib$, and in the first and last steps the changes take place in pairs, with one element increasing by one and the other decreasing by the same amount.  We need to show that $\Phi$ is well defined in that $\up{\pair{b}{e}}\in\cUb(\be)$.  So we need to check~\ref{UC1}--\ref{UC4}.

Conditions~\ref{UC1} and~\ref{UC2} are true because of the bounds and equalities which must be satisfied in steps~\ref{PHC1} and~\ref{PHC3} before making the assignments, and because in~\ref{PHC2} we know that $\phib(\dn{d})$ satisfies~\ref{UCN2}.
To check~\ref{UC3}, we assume $e_i=\de_i+1$ at the end of \ref{PHC3}, and thus also at the end of \ref{PHC2}. Since $\phib(\dn{d})$ satisfies~\ref{UCN2}, we have that $e_{i+1}>0$ after \ref{PHC2}, and thus also after \ref{PHC3}. 
For the other assertion in~\ref{UC3} assume, towards a contradiction, that $b_{i+1}<\al_{i+1}-1$.  Now the value of $b_{i+1}$ was not changed in~\ref{PHC3} because $e_{i+1}>0$ after \ref{PHC2}.  So $a_{i+1}=b_{i+1}<\al_{i+1}-1$, which contradicts the fact that
$a_i=\al_i-1$ for all $i$.  We also handle~\ref{UC4} by contradiction, assuming that
$b_i>0$ but $e_i=0$.  Under these circumstances, $e_i$ would have been reassigned to be $1$ in~\ref{PHC3}.  Thus we have shown that all four conditions for a filter are satisfied.

Finally, we define $\Phib^{-1}$ by~\eqref{eq:Phiinv} with $\Phi$ replaced by $\Phib$.    The demonstration that this is well defined and indeed the inverse of $\Phib$ is much the same as the proof for $\Phi$,  and so left to the reader.
\end{proof}


\section{Comments and open questions}
\label{coq}

This section is devoted to some remarks and a number of open questions which we hope the reader will be interested in pursuing.

\subsection{Extending the bijections}

Even though the map $\Phi:\cI_k(\be)\ra \cU_k(\be)$ from Theorem~\ref{thm:Phi} is not well defined when condition~\eqref{kle} does not hold, it is possible to extend it to any value of $k$ if we restrict the map to a particular subset of ideals, namely those for which $\phi$ can be applied in step \ref{PH2}. 
We continue to use the notation established at the beginning of Subsection~\ref{bpf}.
We say that an ideal of $F(\be)$ encoded by $\dn{\pair{a}{d}}$ is {\em restricted} if, in addition to \ref{IF1}--\ref{IF4}, it satisfies the two conditions:
\begin{enumerate}
    \item[IF5] $d_1\le\de_1$,
    \item[IF6] either $d_\ell\neq1$ or $a_{\ell+1}< \al_{\ell+1}-1$.
\end{enumerate}
Note that $I=\dn{d_1,d_2,\ldots,d_\ell}$ is a restricted ideal of the gate $G(\de_1,\de_2,\ldots,\de_\ell)$ if and only if 
$$
I'=\Dn{\begin{array}{cccccccccc} 1 && 0 &&\cdots&& 0&& 1\\
& d_1 && d_2 &&\cdots&& d_\ell &
\end{array}}
$$
is a restricted ideal of the fence $F(1,\de_1,1,\de_2,\ldots,\de_\ell,1)$.
Indeed, IF5 is the first condition in~\ref{I3}.
And since $\al_{\ell+1}-1=1$, we have $a_{\ell+1}=\al_{\ell+1}-1$, and so IF6 reduces to the second condition in~\ref{I3}.

Similarly, we say that a filter of $F(\be)$ encoded by $\up{\pair{b}{e}}$ is {\em restricted} if, in addition to \ref{UF1}--\ref{UF4}, it satisfies:
\begin{enumerate}
    \item[UF5] $e_\ell\le \de_\ell$,
    \item[UF6] either $e_1\neq 1$ or $b_1<\al_1-1$.
\end{enumerate}
Denote by $\cIr_k(\be)$ and $\cUr_k(\be)$ the subsets of restricted ideals in $\cI_k(\be)$ and restricted filters in $\cU_k(\be)$, respectively. The reader should keep in mind that this notation refers to restricted ideals and filters in fences, not gates.
If $k$ satisfies~\eqref{kle}, then conditions~IF5--IF6
and~UF5--UF6 always hold, and so $\cIr_k(\be)=\cI_k(\be)$ and $\cUr_k(\be)=\cU_k(\be)$ in this case. 

A slight adaptation of the proof of Theorem~\ref{thm:Phi} demonstrates the following.

\begin{thm}\label{thm:Phi-anyk}
Let $\be=(\be_1,\be_2,\ldots,\be_{2\ell+1})$. For any $k$, 
the map $\Phi:\cIr_k(\be)\ra \cUr_k(\be)$  defined by~\ref{PH1}--\ref{PH3} is a bijection.\hqed
\end{thm}


\begin{question} 
Is it possible to give an injective proof of Theorem~\ref{heavy} using a variant of $\Phi$?
\end{question}

For example, Theorem~\ref{thm:Phi-anyk} reduces the problem of comparing the number of ideals and filters of size $k$ to the special case of ideals and filters that fail to satisfy~IF5--IF6 and~UF5--UF6. Ideals that fail to satisfy IF5 (respectively, IF6) are in bijection with ideals of the fence obtained by removing the first (respectively, last) two segments, and similarly for filters.

In a similar vein, we wonder whether it is possible to use a variant of $\Phib$ to give an injective proof of the following conjecture of O\u{g}uz and Ravichandran.

\begin{conj}[\cite{OR:rpf}]
\label{FbUni}
If  $\be=(\be_1,\be_2,\ldots,\be_{2\ell})$ then $\rb(\be)$ is unimodal except when 
 $\be = (1,k,1,k)$ or $(k,1,k,1)$ for some $k\ge1$.
\end{conj}

\subsection{Log-concavity}

Another important property of some real sequences is log-concavity.  Call
$a_0,a_1,\ldots,a_n$ {\em log-concave} if
$$
a_i^2 \ge a_{i-1} a_{i+1}
$$
for all $0<i<n$.  It is well known, and easy to prove, that if a sequence contains only positive reals then log-concavity implies unimodality.  It is not true that $r(\be)$ is always log-concave, as can be seen in the example after Conjecture~\ref{MGO} where $\be=(1,1)$  and $r(\be):1,2,1,1$.  
It is also possible for $\rb(\be)$ to be unimodal, but not log-concave; for example, when $\be=(1,1,1,1,1,1)$, we have $\rb(\be):1,3,3,4,3,1$.  This raises the following question.
\begin{question}
For which $\be$ are  $r(\be)$ or $\rb(\be)$ log-concave?  Even if the whole sequence is not log-concave, is there a long portion of it which is?
\end{question}

\subsection{Chain decompositions}

In~\cite{MSS:ruc}, McConville, Sagan, and Smyth made another conjecture which implies Theorem~\ref{heavy} but remains open.  It has to do with certain chain decompositions of posets.  Let $(P,\lte)$ be a poset.  If $x,y\in P$ then an {\em $x$--$y$ chain} in $P$ is a totally ordered subset $C:x_1\lt x_2\lt\ldots\lt x_l$  with $x=x_1$ and $y=x_l$.
Call $C$ {\em saturated} if $x_{i+1}$ covers $x_i$ for all $1\le i<l$.
A {\em chain decomposition} (CD) of $P$ is a partition $P=\uplus_i C_i$ where the $C_i$ are saturated chains.  

Suppose now that $P$ is ranked with rank function $\rk$.  The {\em center} of a saturated $x$--$y$ chain $C$ is the average
$$
\cen C = \frac{\rk x + \rk y}{2}.
$$
Let $n$ be the maximum rank of an element of $P$.  Call a saturated chain {\em symmetric} if $\cen C = n/2$.  A {\em symmetric chain decomposition} or SCD is a chain decomposition all of whose chains are symmetric.  It is easy to see that if $P$ admits an SCD then its rank sequence is symmetric and unimodal. Having an SCD also implies that $P$ has the strong Sperner property as discussed in the survey article of Greene and Kleitman~\cite{GK:ptt}.  Greene and Kleitman also gave  a famous SCD  of the Boolean algebra of all subsets of a finite set~\cite{GK:svs}.

There is an analogue of SCDs for top and bottom interlacing rank sequences.
As in the previous paragraph, let $P$ be ranked with maximum rank $n$.
Call a chain decomposition of $P$ {\em top centered}, or a {\em TCD}, if for every chain $C$ in the decomposition we have $\cen C = n/2$ or $(n+1)/2$.  Again, a simple argument shows that if $P$ has a TCD then its rank sequence is top interlacing.  Similarly, a {\em bottom centered chain decomposition}, or {\em BCD}, has all chains satisfying $\cen C = n/2$ or $(n-1)/2$.  As expected, this property implies a bottom interlacing rank sequence.  
\begin{conj}[\cite{MSS:ruc}]
The lattice $L(\be)$ admits either an SCD, BCD, or TCD consistent with Theorem~\ref{heavy}.
\end{conj}

McConville, Sagan, and Smyth were able to prove this conjecture using modifications of the Greene-Kleitman SCD whenever $\be$ has a most three parts or is of the form $\be=(k,1,k,1,\ldots,k,1,l)$ for some $1\le l\le k$.
Frustratingly, there seems to be an inductive procedure which always produces a CD of the desired type for $F(\be)$, even though it has not been possible to prove that it always works.  Let $P$ be any finite poset and let $L$ be the corresponding distributive lattice of lower order ideals.  Let $x_1,x_2,\ldots,x_n$ be a linear extension of $P$.  Then any subset of $P$ can be written as an increasing sequence with respect to this extension.  
For example,  the fence $F(2,4,1)$ in Figure~\ref{F(2,4,1)} has linear extension $$
x_7,x_8,x_6,x_5,x_4,x_1,x_2,x_3
$$
which would associate the ideal $I=\{x_1,x_6,x_7,x_8\}$
with the sequence $x_7,x_8,x_6,x_1$.
So any two subsets can now be compared using lexicographic order on their sequences.  A corresponding {\em lexicographic chain decomposition} or {\em LCD} is $L=C_1 \uplus \ldots\uplus C_l$  obtained as follows.  Suppose $C_1,\ldots,C_{i-1}$ have been constructed and let $L'=C_1\uplus \ldots\uplus C_{i-1}$.  We now construct $C_i:I_1\lt I_2\lt\ldots\lt I_j$.  Suppose that the smallest rank of an element of the set difference $L-L'$ is $r$.
Choose the lexicographically smallest element of $L-L'$ having rank $r$ 
to be $I_1$.  Let $I_2$ be the lexicographically smallest element
of $L-L'$ which covers $I_1$.  Continue in this way until it is not possible to pick a covering element of the current ideal for $C_i$ from $L-L'$, at which point the chain terminates.  We iterate this construction until all elements of $L$ are in a chain.
\begin{conj}[\cite{MSS:ruc}]
For any $\be$, there is a linear extension of $F(\be)$ whose corresponding LCD is an SCD, BCD, or TCD of $L(\be)$ consistent with Theorem~\ref{heavy}.
\end{conj}

The difficulty in proving this conjecture is not that it is hard to find such a linear extension.  Indeed, so many linear extensions give a CD of the desired type that it is hard to find a common feature which runs through some subset of them.

\subsection{Distributive lattices}

By the Fundamental Theorem of Finite Distributive Lattices, every distributive lattice $L$ can be obtained as the set of lower order ideals of some poset $P$ ordered by inclusion.  In this case we write $L=L(P)$.  Given 
what has been discussed, the following is a natural question to ask.
\begin{question}
What conditions on a poset $P$ imply that the rank sequence of $L(P)$ satisfies conditions on sequences such as symmetry, unimodality, and so forth?  What conditions on $P$ guarantee that $L(P)$ has an SCD, BCD, or TCD?
\end{question}

\subsection{Rowmotion}

Fences also have connections with dynamical algebraic combinatorics.  Information about this relatively new area of combinatorics can be found in the survey articles of Roby~\cite{rob:dac} or Striker~\cite{str:dac}.
Let $G$ be a group acting on a finite set $S$ with orbits $\cO$.  Consider a {\em statistic} on $S$, which is a map $\st:S\ra\{0,1,2,\ldots\}$. Given a real constant $c$, we say that $\st$ is {\em $c$-mesic} if its average over any orbit $\cO$ is $c$, that is,
$$
\frac{\st\cO}{\#\cO} = c
$$
where $\st\cO = \sum_{x\in\cO} \st x$.

Given any poset $P$, there is a well-studied action called {\em rowmotion} on $L(P)$, viewed as the set of lower order ideals of $P$.  The generator of rowmotion is the map $\rho:L(P)\ra L(P)$ defined as follows.  Given $I\in L(P)$, the antichain $A$ of its maximal elements generates an upper order ideal $U$.  Define $\rho(I)=L(P)-U$.  
Elizalde, Plante, Roby, and Sagan~\cite{EPRS:rf} showed that rowmotion on $L(\be)$ has many interesting properties, but they were unable to resolve the following conjecture.
\begin{conj}
Suppose $k\ge2$ and $\be=(k-1,k,k,\ldots,k,k-1)$ where the number of parts is odd.  For
$I\in L(\be)$, define $\st(I) = \#I$.  Then $\st$ is $n/2$-mesic, where
$n=\#F(\be)$.
\end{conj}



\end{document}